\documentclass{amsart}
\usepackage{amssymb}
\usepackage{enumerate}
\usepackage{url}
\usepackage{amsmath}
\usepackage{graphicx}
\usepackage{amsfonts}
\usepackage{amssymb}
\usepackage{amsthm}
\usepackage{hyperref}

\usepackage{xypic}
\usepackage{graphicx}
\usepackage{pst-plot, pstricks}
\newpsobject{showgrid}{psgrid}{subgriddiv=1,griddots=10,gridlabels=0pt}

\newtheorem{theorem}{Theorem}[section]
\newtheorem{lemma}[theorem]{Lemma}

\newtheorem{proposition}[theorem]{Proposition}

\theoremstyle{definition}
\newtheorem{definition}[theorem]{Definition}

\newtheorem{remark}[theorem]{Remark}

\newcommand\N{\mathbb{N}}
\newcommand\R{\mathbb R}
\newcommand\Z{\mathbb{Z}}

\newcommand\vv{\vee\vee}
\newcommand\V{\mathcal{V}}
\newcommand\Se{\mathcal{S}}

 \newcommand{\ps}[0]{\psi}
 
 \newcommand{\rh}[0]{\varrho}
 \newcommand{\al}[0]{\alpha}
 \newcommand{\be}[0]{\beta}
 \newcommand{\ga}[0]{\gamma}
 
 \newcommand{\ta}[0]{\tau}
 
 \newcommand{\de}[0]{\delta}

 \newcommand{\la}[0]{\lambda}
 \newcommand{\si}[0]{\sigma}
 
 \newcommand{\Ga}[0]{\Gamma}

 \DeclareMathOperator{\supp}{supp}
 \DeclareMathOperator{\PO}{PO}
\DeclareMathOperator{\QM}{QM}

\begin{document}
\title[STABILITY OF QUADRATIC MODULES]
{STABILITY OF QUADRATIC MODULES}
\author{Tim Netzer}
\address{Universit\"at Konstanz, Fachbereich Mathematik und Statistik, 78457 Konstanz, Germany}
\email{tim.netzer@uni-konstanz.de}
\date{\today}

\begin{abstract}A finitely generated quadratic module or
preordering in the real polynomial ring is called \textit{stable},
if it admits a certain degree bound on the sums of squares in the
representation of polynomials. Stability, first defined explicitly
in \cite{ps}, is a very useful property. It often implies that the
quadratic module is closed; furthermore it helps settling the
Moment Problem, solves the Membership Problem for quadratic
modules and allows applications of methods from optimization to
represent nonnegative polynomials.

We provide sufficient conditions for finitely generated quadratic
modules in  real polynomial rings of several variables to be
stable. These conditions can be checked easily. For a certain
class of semi-algebraic sets, we obtain that the nonexistence of
bounded polynomials implies stability of every corresponding
quadratic module. As stability often implies the non-solvability
of the Moment Problem, this complements the result from
\cite{sch}, which uses bounded polynomials to check the
solvability of the Moment Problem by dimensional induction.
 We also use stability
to generalize a result on the Invariant Moment Problem from
\cite{cks}.
\end{abstract}
\maketitle
\section{Introduction} Preorderings and quadratic modules in the real polynomial ring are of great importance in real algebraic geometry.
 They correspond to semi-algebraic sets in a similar way as ideals correspond to algebraic sets. However, it is much more difficult to deal with preorderings
 and quadratic modules than with ideals in general. Nevertheless, substantial progress has been made in this field in the last fifteen years.
 The basic setup is the following. We take finitely many real polynomials $f_{1},\ldots,f_{s}\in\R[\underline{X}]=\R[X_1,\ldots,X_n]$ and consider the basic closed
 semi-algebraic set $$\Se=\Se(f_{1},\ldots,f_{s}):=\left\{x\in\R^{n}\mid f_{1}(x)\geq 0,\ldots,f_{s}(x)\geq 0\right\},$$ as well as the corresponding preordering
 $$\PO(f_{1},\ldots,f_{s}):=\left\{\sum_{e\in\{0,1\}^{s}}\si_{e}f_{1}^{e_{1}}\cdots f_{s}^{e_{s}}\mid \si_{e}\in\R[\underline{X}]^2\right\}$$ and the smaller quadratic module
  $$\QM(f_1,\ldots,f_s)=\left\{\si_0+\si_1f_1+\cdots+\si_sf_s\mid \si_i\in\sum\R[\underline{X}]^2\right\},$$ where $\sum\R[\underline{X}]^2$
 denotes the set of sums of squares of polynomials. Elements from the preordering are obviously nonnegative as polynomial functions on the semi-algebraic set.
  Now one can ask if the preordering or quadratic module contains \textit{all} such nonnegative polynomials. Although this is not true in general, several Positivstellens\"atze give
  representations of nonnegative polynomials. For example, if the semi-algebraic set is compact, the preordering at least contains all strictly positive polynomials,
   by \cite{sch1}. For quadratic modules or noncompact sets, this result fails in general. See for example \cite{m1,pd} for an extensive exposure of the field.

Another question concerns the Moment Problem. We say that the
preordering/ quadratic module $M$ has the \textit{Strong Moment
Property}, if every linear functional on $\R[\underline{X}]$ which
is nonnegative on $M$ is integration with respect to a measure on
the corresponding semi-algebraic set.  The result from \cite{sch1}
implies that every preordering describing a compact semi-algebraic
set has the Strong Moment Property, and \cite{sch} gives a
criterion for the case of a noncompact set, see also \cite{m1,n}.

If one already knows that a polynomial belongs to the preordering
or quadratic module, it is another problem how to find  an
explicit sums of squares representation. In general, the degree of
the sums of squares used in the representation of some $f$ can not
be bounded by a function that only depends on  the degree of $f$.
For example, in the case of a compact set $\Se,$ one has to take
into account the degree, the size of the coefficients and the
minimum of $f$ on $\Se$, so be able to say something about the
degree of the sums of squares (see \cite[Theorem 8.4.3]{pd} and
\cite{sw2}). This is what makes it so difficult to find
representations.

 Now the notion of \textit{stability} of a finitely generated preordering or quadratic module has first been introduced
explicitly in \cite{ps}.  In the polynomial ring, stability means
that every polynomial in the preordering has a representation,
where the degree of the sums of squares can be bounded by a number
depending only on the degree of the polynomial. The authors of
\cite{ps} give a strong geometric criterion for  quadratic modules
to be stable. Roughly speaking, if the set $\Se$ is big enough at
infinity, then every corresponding finitely generated quadratic
module is stable. The notion has also been dealt with in
\cite{p1,p2}, where the geometric result from \cite{ps} is applied
and extended, for curves and surfaces mostly.

The importance of stability is evident from several results.
First, as shown in \cite{ps}, stable quadratic modules are often
closed (with respect to the finest locally convex topology). This
was also shown in \cite{km}, Theorem 3.5, in the case that $\Se$
contains a full dimensional cone, but without using the notion of
stability explicitly. Similar arguments have been used in
\cite{pd}, Proposition 6.4.5., and \cite{sch0}, Section 11.6.

Second, stability often excludes the Strong Moment Property of
quadratic modules. This useful fact was shown in \cite{s},
generalizing an idea by Prestel and Berg.  The result also shows
that one can often not expect finitely generated quadratic modules
to be stable.

A further reason making stable quadratic modules  so interesting
is that the degree bound condition allows the application of model
theoretic methods. Indeed, the set of all polynomials of fixed
degree which lie in the quadratic module can be defined by a first
order logic formula then. This also solves the so called
\textit{Membership Problem} for stable quadratic modules. Whether
the membership problem is solvable for arbitrary quadratic modules
is still an open question. So far it is only known for finitely
generated preorderings in the real polynomial ring of one
variable, see \cite{au}.

Also the question of finding an explicit representation of a
polynomial in a stable quadratic module is easy to solve. Indeed,
it can be translated into a semi-definite programming problem,
which can be solved efficiently. See \cite{l,sw1,vb} for details.

Our contribution is the following. We define the notion of
\textit{stability with respect to a grading}, for quadratic
modules (see Section \ref{definitions}). This notion of stability
has a characterization which is of purely geometric nature (see
Section \ref{character}). We then relate it to the notion of
stability used in \cite{ps,s}. Indeed, this is the stability one
is mostly interested in.  Our results allow to obtain this
stability in a lot of cases by checking some easy geometric or
combinatorial properties (see Section \ref{sec} and the explicit
examples in Section \ref{examples}).

For a certain class of semi-algebraic sets we are able to proof
that the absence of nontrivial bounded polynomials implies the
stability of every corresponding finitely generated quadratic
module (Theorem \ref{all}). Thus no such quadratic module can have
the Strong Moment Property. This complements the result from
\cite{sch}, that uses bounded polynomials to check the Strong
Moment Property by dimensional induction.

Last, we use the notion of \textit{strong stability} to improve
upon a result from \cite{cks}, while simplifying the proof. This
is done in the last section. \vspace{2ex}

 {\bf Acknowledgements:} I wish to thank  Daniel
Plaumann and Markus Schweighofer  for many interesting and helpful
discussions on the topic. Financial support by the Studienstiftung
des deuschen Volkes is greatfully acknowledged.

\section{Notations and Preliminaries}
For this whole work, let $A=\R[\underline{X}]=\R[X_1,\ldots,X_n]$
be the real polynomial algebra in $n$ variables. A subset
$M\subseteq A$ is called a \textit{quadratic module}, if
$$1\in M,\ M+M\subseteq M \mbox{  and  } A^{2}\cdot M\subseteq M$$ holds, where
$A^{2}$  denotes the set of squares in $A$. For elements
$f_{1},\ldots,f_{s}\in A$, $\QM(f_{1},\ldots,f_{s}),$ defined in
the introduction, is the smallest quadratic module containing
$f_{1},\ldots,f_{s}$. It is called the quadratic module
\textit{generated by} $f_{1},\ldots,f_{s}$. We always assume
 generators of a quadratic module to be all non-zero.
A quadratic module is called a \textit{preordering}, if it is
closed under multiplication, i.e. if $M\cdot M\subseteq M$ holds.
For $f_{1},\ldots,f_{s}\in A$, $\PO(f_{1},\ldots,f_{s})$, again as
in the introduction,
 is the smallest
preordering containing $f_{1},\ldots,f_{s}$. It is called the
preordering \textit{generated by} these elements. For any
quadratic module $M$, $\supp(M):= M\cap -M$ is called the
\textit{support of $M$}. It is an ideal of $A$. For a quadratic
module $M$ in $A$ write
$$\Se(M):=\left\{x\in\R^n\mid f(x)\geq 0 \mbox{ for all } f\in
M\right\}.$$ Of special interest is the case that $M$ is finitely
generated. If $f_1,\ldots,f_s$ are generators of $M$, then
$\Se(M)=\left\{x\in\R^n\mid f_1(x)\geq 0,\ldots,f_s(x)\geq
0\right\}$ is called \textit{basic closed semi-algebraic}. We
include the proof of the following proposition due to the lack of
a good reference.

\begin{proposition}\label{sper}
Let $M$ be a finitely generated quadratic module in $A$. If
$\Se(M)$ is Zariski-dense in $\R^n$, then $\supp(M)=\{0\}$. If $M$
is a finitely generated preordering, then $\supp(M)=\{0\}$ implies
the Zariski-denseness of $\Se(M)$.
\end{proposition}
\begin{proof}
Take $f\in\supp(M)$. Then $f=0$ on $\Se(M)$, so $f=0$  by the
Zariski-denseness.

Now suppose $M$ is a finitely generated preordering with
$\supp(M)=\{0\}$. Suppose $f=0$ on $\Se(M)$ for some $f\in A$. By
Theorem 4.2.11 from \cite{pd} there are $t_1,\tilde{t}_1,
t_2,\tilde{t}_2\in M$ and $e_1,e_2\in\N$ such that $t_1f=f^{2e_1}
+t_2$ and $\tilde{t}_1(-f)=f^{2e_2}+\tilde{t}_2$ holds. So
$t_1\tilde{t}_1f\in\supp(M)$, so $f=0$. This shows the desired
denseness.
\end{proof}

Following \cite{s}, for any $\R$-subspace $W$ of $A$ we write
$\sum\left(W;f_1,\ldots,f_s\right)$ for the set of all elements
$$\si_0+\si_1f_1+\cdots + \si_sf_s,$$ where $\si_i\in\sum W^2$ for
all $i$. Obviously each $\sum(W;f_1,\ldots,f_s)$ is contained in
$M=\QM(f_1,\ldots,f_s)$ and $\sum(A;f_1,\ldots,f_s)=M$. If $W$ is
finite dimensional, then $\sum(W;f_1,\ldots,f_s)$ is contained in
a finite dimensional subspace of $A$. The following definition is
Definition 3.2 from \cite{s}:

\begin{definition}\label{staborig}$M=\QM(f_1,\ldots,f_s)$ is called
\textit{stable}, if for every finite dimensional subspace $U$ of
$A$ there is another finite dimensional subspace $W$ of $A$ such
that $$M\cap U\subseteq\sum(W;f_1,\ldots,f_s)$$ holds.
\end{definition}

The following two results show the importance of the notion:

\begin{theorem}[Powers, Scheiderer \cite{ps}]
If $M$ is stable and $\Se(M)$ is Zariski-dense in $\R^n$, then $M$
is closed, i.e. $M=M^{\vv}$ holds, where $M^{\vv}$ denotes the
double dual cone of $M$.
\end{theorem}
\begin{theorem}[Scheiderer \cite{s}] If $M$ is stable and $\Se(M)\subseteq\R^n$
has dimension at least two, then $M$ does not have the Strong
Moment Property. In particular, $M$ does not contain all
polynomials that are nonnegative on $\Se(M)$.
\end{theorem}

For our approach towards stability, we need the notions of
filtrations and gradings. So let $(\Gamma,\leq)$ be an ordered
Abelian group, i.e. an Abelian group $\Gamma$ with a linear
ordering, such that $\alpha\leq\beta \Rightarrow \alpha + \gamma
\leq\beta+\gamma$ holds for any $\alpha,\beta,\gamma\in\Gamma$.

\begin{definition}
A \textit{filtration} of $A$ is a family
$\{U_{\gamma}\}_{\gamma\in\Gamma}$ of linear $\R$-subspaces of
$A$, such that for all $\ga,\ga'\in\Ga$
$$\gamma\leq\gamma' \Rightarrow U_{\gamma}\subseteq U_{\gamma'},$$
$$U_{\ga}\cdot U_{\ga'}\subseteq U_{\ga+\ga'},$$
$$\bigcup_{\ga\in\Ga}U_{\ga} = A \mbox{ and }$$ $$1\in U_0$$  holds.
\end{definition}

\begin{definition}
A \textit{grading} of $A$ is a decomposition of the $\R$-vector
space $A$ into a direct sum of linear subspaces:
$$A=\bigoplus_{\ga\in\Ga}A_{\ga},$$ such that $A_{\ga}\cdot
A_{\ga'}\subseteq A_{\ga+\ga'}$ holds for all $\ga,\ga'\in\Ga$.
\end{definition}Any element $0\neq f\in A$ can then be written in a unique way as
$$f=f_{\ga_{1}}+\cdots + f_{\ga_{d}}$$ for some $d\in\N$ and $0\neq
f_{\ga_{i}}\in A_{\ga_{i}}$, where $\ga_{1} < \ga_{2} < \cdots <
\ga_{d}$. Then $\deg(f):=\ga_{d}$ is called the \textit{degree of
$f$}, and $f^{\max}:=f_{\ga_{d}}$ is called the \textit{highest
degree part of $f$}. Elements from $A_{\ga}$ are called
\textit{homogeneous of degree $\ga$}. The degree of $0$ is
$-\infty$. One easily checks that $1\in A_0.$

The following are some easy observations: If
$A=\bigoplus_{\ga\in\Ga}A_{\ga}$ is a grading, then
$$U_{\tau}:= \bigoplus_{\ga\leq\tau}A_{\ga}$$ defines a filtration
$\{U_{\tau}\}_{\tau\in\Ga}$ of $A$. If  $\nu\colon K \rightarrow
\Ga\cup\{\infty\}$ is a valuation of the quotient field
$K=\R(X_1,\ldots,X_n)$ of $A$ which is trivial on $\R$, then
$$U_{\ga}:=\{ f\in A\mid \nu(f)\geq -\ga\}$$ defines a filtration
$\{U_{\ga}\}_{\ga\in\Ga}$ of $A$. If
$A=\bigoplus_{\ga\in\Ga}A_{\ga}$ is a grading, then
$$\nu\left(\frac{f}{g}\right):= \deg(g)-\deg(f)$$ defines a valuation on the
quotient field $K$, trivial on $\R$. This valuation induces the
same filtration on $A$ as the grading. For any grading and all
$f,g\in A$ we have $\deg(f\cdot g)=\deg(f)+\deg(g)$  and
$$\deg(f^2+g^2)=\max\{ \deg(f^2),\deg(g^2)\}
=2\max\{\deg(f),\deg(g)\}.$$

\section{Definitions of Stability}\label{definitions}

\begin{definition}\label{stabdef} Let $\{U_{\gamma}\}_{\ga\in\Ga}$ be a
filtration of $A$ and $f_{1},\ldots,f_{s}$ generators of the
quadratic module $M$. We set $f_{0}=1$.

(1) $f_{1},\ldots,f_{s}$ are called \textit{stable generators of
$M$ with respect to the filtration}, if there is a monotonically
increasing map  $\rh\colon \Ga \rightarrow\Ga$, such that
$$M\cap U_{\ga}\subseteq
\sum\left(U_{\rh(\ga)};f_1,\ldots,f_s\right)$$ holds for all
$\ga\in\Ga$.

(2) $f_{1},\ldots,f_{s}$ are called \textit{strongly stable
generators of $M$ with respect to the filtration}, if there is a
monotonically increasing map $\rh\colon \Ga \rightarrow\Ga$, such
that for all sums of squares $\si_{0},\ldots,\si_{s}$, where
$\si_{i}=g_{i,1}^{2}+\cdots+g_{i,k_{i}}^{2}$, we have
$$\sum_{i=0}^{s}\si_{i}f_{i}\in U_{\ga}\Rightarrow g_{i,j}\in U_{\rh(\ga)} \mbox{ for all }i,j.$$
\end{definition}

Obviously, strongly stable generators of $M$ are stable generators
of $M$.  The notion of strong stability has also been introduced
in \cite{p1}, but under a different name. The following Lemma is
essentially the same as \cite{ps}, Lemma 2.9.
\begin{lemma}\label{ind} If $M$ has stable generators with respect to a given filtration, then any finitely many generators of $M$ are
stable generators with respect to that filtration.
\end{lemma}
\begin{proof}Suppose $f_{1},\ldots,f_{s}$ are stable generators of $M$ with stability map $\rh$ as in Definition
\ref{stabdef}(1). Let $g_{1},\ldots,g_{t}$ be arbitrary generators
of $M$. Then we find representations $$f_{i}= \sum_{j=0}^{t}
\si_{j}^{(i)}g_{j},$$ where all $\si_{j}^{(i)}\in\sum
\left(U_{\ta}\right)^{2}$ for some big enough $\ta\in\Ga$. Now
take $f\in M\cap U_{\ga}$ for some $\ga$ and find a representation
$f=\sum_{i=0}^{s}\si_{i}f_{i}$ with $\si_{i}\in\sum \left(
U_{\rh(\ga)}\right)^{2}$ for all $i$. Then
$$f=\sum_{i}\si_{i}f_{i}
=\sum_{i}\si_{i}\sum_{j}\si_{j}^{(i)}g_{j} = \sum_{j}
\left(\sum_{i}\si_{i}\si_{j}^{(i)}\right)g_{j},$$ and all
$\sum_{i}\si_{i}\si_{j}^{(i)}$ are in
$\sum\left(U_{\rh(\ga)+\ta}\right)^{2}$. This shows that
$g_{1},\ldots,g_{t}$ are also stable generators of $M$, with
 stability map $\ga\mapsto \rh(\ga)
+\tau$.
\end{proof}
So it makes sense to talk about stability of a finitely generated
quadratic module with respect to a filtration, without mentioning
the generators. However, the stability map $\rh$ may depend on the
generators in general.

Note that $M$ is stable in the usual sense (defined in the
previous section), if and only if it is stable with respect to a
filtration consisting of finite dimensional subspaces $U_{\ga}$ of
$A$.

Now suppose we are given a grading on $A$.  We will talk about
stable generators, strongly stable generators and stable quadratic
modules \textit{with respect to the grading}, and always mean
these notions with respect to the induced filtration. However,
things become easier to handle in this case.

\begin{lemma}\label{unab}  Let $A=\bigoplus_{\ga\in\Ga}A_{\ga}$ be
a grading and let $M$ be a finitely generated quadratic module in
$A$. Then $M$ has strongly stable generators with respect to the
grading if and only if there is a monotonically increasing map
$\ps\colon\Ga\rightarrow\Ga$, such that for all $f,g\in M$
$$\deg(f)\leq\ps\left(\deg(f+g)\right)$$ holds. In particular, if $M$ has strongly stable generators, then any finitely many generators are strongly stable generators.\end{lemma}
\begin{proof}
Suppose $f_{1},\ldots,f_{s}$ are strongly stable generators of $M$
with stability map $\rh$. Take $f,g$ from $M$ with representations
$f=\sum_{i}\si_{i}f_{i}, g=\sum_{i}\ta_{i}f_{i}$. Then for all
 $j$ \begin{align*}\deg\left(\si_{j}f_{j}\right)&= \deg\left(\si_{j}\right)+\deg\left(f_{j}\right) \\ &\leq \deg\left(\si_{j}+\ta_{j}\right)+\deg(f_{j})\\
&=\deg(\left(\si_{j}+\ta_{j}\right)f_{j})\\ &\leq \ps\left(\deg\left(\sum_{i}\left(\si_{i}+\ta_{i}\right)f_{i}\right)\right)\\
&= \ps\left(\deg\left(f+g\right)\right),\end{align*} where the
last inequality is fulfilled with
$$\ps(\ga):=2\rh(\ga)+\max_{i}\deg(f_{i}),$$ by the strong stability
of the $f_{i}$. So $\deg(f)\leq \ps\left(\deg(f+g)\right)$ holds.
Note that $\ps$ is monotonically increasing, as $\rh$ was.

So now suppose $\deg(f),\deg(g)\leq\ps\left(\deg(f+g)\right)$ for
 some suitable map $\ps$ and all $f,g\in M$. Take any finitely many
(non-zero) generators $f_{1},\ldots,f_{s}$ and sums of squares
$\si_{0},\ldots,\si_{s}$, where $\si_{j}=p_{j,1}^{2}+\cdots
+p_{j,k_{j}}^{2}$. Set $f_{0}=1$. Then
$$\deg\left(\si_{j}f_{j}\right) \leq
\ps\left(\deg\left(\sum_{i}\si_{i}f_{i}\right)\right)$$ for all
$j$. Thus for all $j,l$,
$$2\deg(p_{j,l})\leq\ps\left(\deg\left(\sum_{i}\si_{i}f_{i}\right)\right)-\min_{i}\deg(f_{i}).$$
So $$\deg\left(p_{j,l}\right)\leq \max\left\{ 0,\ps\left(\deg
\left(\sum_{i}\si_{i}
f_{i}\right)\right)-\min_{i}\deg(f_{i})\right\}$$ holds. Now
$\rh(\ta):=\max\left\{
0,\ps\left(\ta\right)-\min_{i}\deg(f_{i})\right\}$ defines a
monotonically increasing map, and whenever $$f=\sum_i \si_if_i\in
\bigoplus_{\ga\leq\ta}A_{\ga} \mbox{ for some }\ta,$$ then
$\deg(p_{j,l})\leq\rh(\deg(f))\leq\rh(\ta),$ which shows the
strong stability of the $f_{1},\ldots,f_{s}$. We have used the
fact that $\rh$ is monotonically increasing in the last
inequality.

The proof shows that any finitely many generators are strongly
stable generators in this case.
\end{proof}So we can talk about  strong stability of a finitely generated quadratic module with
respect to a grading, without mentioning the generators. A very
special case of strong stability is the following, which will have
a nice characterization below.
\begin{definition}
Let $A=\bigoplus_{\ga\in\Ga}A_{\ga}$ be a grading and let
$M\subseteq A$ be a finitely generated quadratic module. $M$ is
\textit{totally stable with respect to the grading}, if
$$\deg(f) \leq\deg(f+g)$$ holds for all $f,g\in M$.
The proof of Lemma \ref{unab} shows that this is equivalent to the
fact that there are generators $f_1,\ldots, f_s$ of $M$ such that
$$\deg(\si_jf_j)\leq \deg\left(\sum_{i}\si_if_i\right) $$
holds for all $\si_j\in\sum A^2.$ Any finite set of generators of
$M$ fulfills this condition then.
\end{definition}

Note that a quadratic module $M$ in $A$ which is totally stable
with respect to a grading has trivial support. Indeed if $f,-f\in
M$, then $\deg(f)\leq \deg(f-f)=\deg(0)=-\infty,$ so $f=0$.

If $\nu\colon K\rightarrow\Ga\cup\{\infty\}$ is the valuation
corresponding to a given grading, then the notion of total
stability is equivalent to saying that for any $f,g\in M$,
$$\nu(f+g)=\min\left\{ \nu(f),\nu(g)\right\}$$ holds. This is
usually called \textit{weak compatibility of $\nu$ and $M$}.

\section{Characterizations of Stability}\label{character}

Total stability with respect to a grading turns out to be well
accessible. First, when checking total stability of a finitely
generated quadratic module, one can apply an easy reduction
result, to obtain possibly smaller quadratic modules. Therefore
take generators $f_1,\ldots,f_s$ of $M$, define an equivalence
relation on the generators by saying
$$f_i\equiv f_j :\Leftrightarrow \deg(f_i)\equiv \deg(f_j) \mod
2\Ga,$$ and group them into equivalence classes
$$\left\{f_{i1},\ldots,f_{is_i} \right\}\ (i=1,\ldots,r).$$ Then
total stability reduces to total stability of the quadratic
modules generated by these equivalence classes:

\begin{proposition}\label{modtwo}
$M$ is totally stable with respect to the given grading if and
only if all the quadratic modules
$$M_i:=\QM(f_{i1},\ldots,f_{is_i})$$ are totally stable.
\end{proposition}
\begin{proof}
The "only if"-part it obvious. For the "if"-part take $g,h\in M$
with representations $g=\si_0 +\si_1f_1 +\cdots +\si_sf_s$ and $h
=\ta_0 + \ta_1 f_1 + \cdots +\ta_s f_s$. By grouping the terms
with respect to the equivalence relation and using the total
stability of the modules $M_i$, we get decompositions
$$g=g_1 +\cdots + g_r,\quad h=h_1 + \cdots + h_r$$ with
$g_i,h_i\in M_i$ and all the $g_i$ (as well as the $h_i$) have a
different degree modulo $2\Ga$. So if $g$ and $h$ have the same
degree and $\deg(g)=\deg(g_k),\deg(h)=\deg(h_l)$, then $k=l$ and
the highest degree parts of $g$ and $h$ cannot cancel out, due to
the total stability of  $M_k$.
\end{proof}

Now total stability has the following easy characterization:
\begin{proposition}\label{supp} Let $A=\bigoplus_{\ga\in\Ga}A_{\ga}$ be a grading and let $M$ be a finitely generated quadratic module in $A$.
Let $f_1,\ldots,f_s$ be generators of $M$. Then
$$M \mbox{ is totally stable} \Leftrightarrow \supp\left(\QM(f_1^{\max},\ldots,f_s^{\max})\right)=\{0\}.$$
\end{proposition}
\begin{proof}
First suppose
$\supp\left(\QM(f_1^{\max},\ldots,f_s^{\max})\right)\neq\{0\}.$ So
there are sums of squares $\si_{0},\ldots,\si_{s}$, not all zero,
 such that
$\sum_{i=0}^{s}\si_{i}f_{i}^{\max}=0$. Now
\begin{align*}\deg\left(\sum_{i=0}^{s}\si_{i}f_{i} \right)&=\deg\left(\sum_{i=0}^{s}\si_{i}(f_{i}-f_{i}^{\max})\right)\\
&\leq \max_{i}\left\{ \deg(\si_{i}(f_{i} -f_{i}^{\max})
\right\}\\&< \max_{i}\left\{ \deg(\si_{i}f_{i})\right\},
\end{align*}so $M$ is not totally stable.
Conversely, for any sum of squares $\si_j$, the highest degree
part of $\si_jf_j$ lies in $\QM(f_1^{\max},\ldots,f_s^{\max}).$ So
when adding elements of the form $\si_if_i$, the highest degree
parts cannot cancel out, if
$\supp\left(\QM(f_1^{\max},\ldots,f_s^{\max})\right)=\{0\}$. So
$M$ is totally stable.\end{proof}

The good thing about Proposition \ref{supp} is, that it allows to
link total stability to a geometric condition, via Proposition
\ref{sper}:

\begin{theorem}\label{stabmain}
Let $A=\bigoplus_{\ga\in\Ga}A_{\ga}$ be a grading and $M$ a
finitely generated quadratic module in $A$. If for a set of
generators $f_1,\ldots,f_s$ of $M$, the set
$$\Se(f_1^{\max},\ldots,f_s^{\max})\subseteq \R^n$$ is Zariski dense, then
$M$ is totally stable with respect to the grading. If $M$ is
closed under multiplication, then total stability implies the
Zariski denseness for any finite set of generators of $M$.
\end{theorem}
\begin{proof} If $\Se(f_1^{\max},\ldots,f_s^{\max})$ is Zariski
dense, then
$$\supp\left(\QM(f_1^{\max},\ldots,f_s^{\max})\right)=\{0\},$$ by
Proposition \ref{sper}. So Proposition \ref{supp} yields the total
stability of $M$. If $M$ is a preordering, generated by
$f_1,\ldots,f_s$ as a quadratic module, and totally stable, then
$\QM(f_1^{\max},\ldots,f_s^{\max})$ is also a preordering. So
Propositions \ref{supp} and \ref{sper} imply the Zariski-denseness
of $\Se(f_1^{\max},\ldots,f_s^{\max})$ in $\R^n$.
\end{proof}

Note that if $M$ is a finitely generated quadratic module which is
closed under multiplication, and $f_1,\ldots,f_t$ generate $M$ as
a \textit{preordering}, then the products $f_e:=f_1^{e_1}\cdots
f_t^{e_t}$ ($e\in\{0,1\}^t$) generate $M$ as a quadratic module,
and
$$\Se(f_1^{\max},\ldots,f_t^{\max})=\Se\left(f_e^{\max}\mid
e\in\{0,1\}^t\right).$$

In the next section we will consider different kinds of gradings
on $A$. The denseness condition from Theorem \ref{stabmain} will
be translated into a geometric condition on the original set
$\Se(M)$.

Recall that we are mostly interested in stability of a finitely
generated quadratic module in the sense of \cite{ps} (see
Definition \ref{staborig}), that is, stability with respect to a
filtration of finite dimensional subspaces. Many of the later
considered gradings do \textit{not} induce such finite dimensional
filtrations. Our goal is then to find stability with respect to
enough different gradings, so that in the end the desired
stability is still obtained. Therefore we consider the following
setup: Let $\Ga,\Ga_{1},\ldots,\Ga_{m}$ be ordered Abelian groups
and let
$$\left\{W_{\ga}\right\}_{\ga\in\Ga},\left\{U_{\ga}^{(j)}\right\}_{\ga\in\Ga_{j}}
(j=1,\ldots,m)$$ be filtrations of $A$.

\begin{definition}\label{cover} The filtration $\left\{W_{\ga}\right\}_{\ga\in\Ga}$ is \textit{covered} by  the
filtrations $$\left\{U_{\ga}^{(j)}\right\}_{\ga\in\Ga_{j}}
(j=1,\ldots,m),$$ if there are monotonically increasing maps
$$\eta\colon\Ga_{1} \times\cdots\times\Ga_{m}\rightarrow\Ga, \
\eta_{j}\colon\Ga\rightarrow\Ga_{j} \ (j=1,\ldots,m), $$ such that
for all $\ga\in\Ga,\ga_{j}\in\Ga_{j} \ (j=1,\ldots,m)$, the
following holds:
$$W_{\ga}\subseteq\bigcap_{j=1}^{m} U_{\eta_{j}(\ga)}^{(j)} \quad \mbox{ and }$$ $$\bigcap_{j=1}^{m}U_{\ga_{j}}^{(j)}
\subseteq W_{\eta(\ga_{1},\ldots,\ga_{m})}.$$
\end{definition}
For $\eta$, \textit{monotonically increasing} refers to the
\textit{partial} ordering on the product group obtained by the
componentwise orderings of the factors.

 We will speak about covering of/by
\textit{gradings}, and mean the notion from Definition \ref{cover}
applied to the induced filtrations. The next theorem makes clear
why we are interested in coverings.

\begin{theorem}\label{cost}
Suppose a quadratic module $M$ in $A$ has generators
$f_{1},\ldots,f_{s},$ which are strongly stable generators with
respect to all the filtrations
$$\left\{U_{\ga}^{(j)}\right\}_{\ga\in\Ga_{j}} (j=1,\ldots,m).$$
Then $f_{1},\ldots,f_{s}$ are also strongly stable generators of
$M$ with respect to any filtration
$\left\{W_{\ga}\right\}_{\ga\in\Ga}$ which is covered by these
filtrations.
\end{theorem}
\begin{proof} For every $j=1,\ldots,m$, take a stability map $\rh_{j}$ for the generators with
respect to the filtration
$\left\{U_{\ga}^{(j)}\right\}_{\ga\in\Ga_{j}}$ (remember
Definition $\ref{stabdef}(2)$). As in Definition $\ref{cover}$ ,
the covering maps are denoted by $\eta$ and $\eta_{j}$.

Take sums of squares $\si_{0},\ldots,\si_{s}$, where
$\si_{i}=g_{i,1}^{2}+\cdots+g_{i,k_{i}}^{2}$ and suppose
$\sum_{i=0}^{s}\si_{i}f_{i}\in W_{\ga}$ for some $\ga\in\Ga$. Then
$\sum_{i=0}^{s}\si_{i}f_{i}\in U_{\eta_{j}(\ga)}^{(j)}$ for all
$j$. So by strong stability, $$g_{i,l}\in
U_{\rh_{j}\left(\eta_{j}(\ga)\right)}^{(j)}\mbox{ for all }
j,i,l.$$ But then $$g_{i,l}\in
W_{\eta\left(\rh_{1}(\eta_{1}(\ga)),\ldots,\rh_{m}(\eta_{m}(\ga))
\right)}\mbox{ for all }i,l,$$ which shows the strong stability
with respect to $\left\{W_{\ga}\right\}_{\ga\in\Ga}$.
\end{proof}
So we are taking the following approach towards stability in the
sense of \cite{ps}: First we use Theorem \ref{stabmain} for enough
different gradings on $A$, to obtain conditions for total (and
therefore strong) stability of a quadratic module with respect to
each of the gradings. If the gradings are chosen in the right way,
Theorem \ref{cost} yields total stability with respect to a
filtration of finite dimensional subspaces, and therefore
stability in the sense of \cite{ps}.

\begin{remark} One checks that all the results hold in
more general algebras than the polynomial algebra over $\R$.
Indeed, for any real closed field $R$ and any finitely generated
$R$-algebra that is a real domain, the results remain valid.
\textit{Real} means, that a sum of squares $a_1^2+\cdot +a_t^2$ in
$A$ can only be zero if all $a_i$ are zero. $A$ is called a
\textit{domain}, if it does not contain zero divisors. Note that
we have used these two properties at several points in the
previous proofs.

The notions of stable and strongly stable generators with respect
to a filtration even make sense in arbitrary $R$-algebras. We come
back to this in the last section of the paper, where we will
generalize a result from \cite{cks}.

\end{remark}

\section{Examples of Gradings and Applications}\label{sec}

As above, let $A=\R[\underline{X}]=\R[X_1,\ldots,X_n]$ be the real
polynomial algebra in $n$ variables. For
$\de=(\de_1,\ldots,\de_n)\in\N^{n}$ and
$z=(z_1,\ldots,z_n)\in\Z^n$ we write
$$\underline{X}^{\de}:= X_1^{\de_1}\cdots X_n^{\de_n}$$ and $$z\cdot \de:=
z_1\de_1 + \cdots + z_n\de_n.$$ For $d\in\Z$ define
$$A_d^{(z)}:=\left\{ \sum_{\de\in\N^n,\ z\cdot\de=d}
c_{\de}\underline{X}^{\de}\mid c_{\de}\in\R\right\}.$$ Then
$$A=\bigoplus_{d\in\Z} A_d^{(z)}$$ is a grading indexed
in the ordered group $(\Z,\leq)$, to which we will refer to as the
\textit{$z$-grading}. For example, $z=(1,\ldots,1)$ gives rise to
the usual degree-grading on $A$, whereas $z=(1,0,\ldots,0)$
defines the grading with respect to the usual degree in $X_1$.
Note that the filtration induced by such a $z$-grading consists of
finite dimensional linear subspaces of $A$ if and only if all
entries of $z$ are positive.

We want to characterize the denseness condition from Theorem
$\ref{stabmain}$ for these $z$-gradings. For a compact set $K$ in
$\R^n$ with nonempty interior, we define the \textit{tentacle of
$K$ in direction of $z$} in the following way:
$$T_{K,z}:=\left\{(\la^{z_{1}}x_{1},\ldots,\la^{z_{n}}x_{n})\mid
\la\geq 1,\ x=(x_{1},\ldots,x_{n})\in K\right\}.$$ For
$z=(1,\ldots,1)$, such a set is just a full dimensional cone in
$\R^n$. For $z=(1,0,\ldots,0)$ it is a full dimensional cylinder
going to infinity in the direction of $x_1$. For
$z=(1,-1)\in\Z^2$, something like the set defined by $xy \leq 2,
xy \geq 1$ and $x \geq 1$ would be such a set.

\begin{proposition}\label{xxx} Let $f_{1},\ldots,f_{s}$ be polynomials in the graded polynomial algebra $A=\bigoplus_{d\in \Z}A_{d}^{(z)}$,
where $z\in \Z^{n}$. Then the set
$$\Se(f_{1}^{\max},\ldots,f_{s}^{\max})\subseteq \R^n$$ is
Zariski-dense in $\R^{n},$ if and only if the set
$$\Se(f_{1},\ldots,f_{s})\subseteq \R^n$$ contains a tentacle $T_{K,z}$ for some
compact $K\subseteq \R^{n}$ with nonempty interior.
\end{proposition}
\begin{proof}
First suppose $\Se(f_{1}^{\max},\ldots,f_{s}^{\max})$ is
Zariski-dense, which is equivalent to saying that there is a
compact set $K$ with nonempty interior, on which all
$f_{i}^{\max}$ are positive. Write each $f_{i}$ as a sum of
homogeneous elements (with respect to the $z$-grading), for
example
$$f_{1}=h_{d_{1}}+\ldots+h_{d_{t}},$$ where $d_{1}<\ldots < d_{t}$
and $0\neq h_{d_{j}}\in A_{d_{j}}^{(z)}$. Then for $x\in \R^{n}$
and $\la>0$
$$f_{1}(\la^{z_{1}}x_{1},\ldots,\la^{z_{n}}x_{n})=\la^{d_{1}}h_{d_{1}}(x)+\ldots+
\la^{d_{t}}h_{d_{t}}(x).$$ As $h_{d_{t}}(x)=f_{1}^{\max}(x)>0$ if
$x$ is taken from $K$, the expression is positive for $\lambda\geq
N$  with $N$ big enough. Thereby $N$ can be chosen to depend only
on the size of the coefficients $h_{d_{j}}(x)$. So $N$ can be
chosen big enough to make
$f_{i}(\la^{z_{1}}x_{1},\ldots,\la^{z_{n}}x_{n})$ positive for all
$\la\geq N$, $x\in K$ and all $i=1,\ldots,s.$ Replacing $K$ by
$$K':=\left\{ (N^{z_1}x_1,\ldots,N^{z_n}x_n)\mid x=(x_1,\ldots,x_n)\in
K\right\}$$ we find $T_{K',z}\subseteq \Se(f_{1},\ldots,f_{s})$.

Conversely, suppose $\Se(f_{1},\ldots,f_{s})$ contains a tentacle
$T_{K,z}$. Then all the highest degree parts of the $f_{i}$ must
be nonnegative on $K$, with the same argument as above. So
$\Se(f_{1}^{\max},\ldots,f_{s}^{\max})$ contains $K$ and is
therefore Zariski-dense in $\R^{n}$.
\end{proof}

Combined with Theorem $\ref{stabmain}$ we get:
\begin{theorem}\label{denseex}
Let $f_{1},\ldots,f_{s}$ be polynomials in the graded polynomial
algebra $A=\bigoplus_{d\in \Z}A_{d}^{(z)}$, where $z\in \Z^{n}$.
If the set $$\Se(f_{1},\ldots,f_{s})\subseteq \R^n$$ contains some
tentacle $T_{K,z}$ ($K$ compact with nonempty interior), then the
quadratic module $M=\QM(f_{1},\ldots,f_{s})$ is totally stable. If
$M=\QM(f_{1},\ldots,f_{s})$ is closed under multiplication, then
$\Se(f_1,\ldots,f_s)$ must contain such a tentacle for $M$ to be
totally stable.
\end{theorem}

For the $z$-gradings, we can also settle the questions of
coverings:

\begin{proposition}\label{qcover} Let $z,z^{(1)},\ldots,z^{(m)}\in\Z^{n}$ and assume there exist numbers \linebreak $r_{1},\ldots,r_{m}$,$t_{1},\ldots,t_{m}\in\N,$
such that the following conditions hold (where $ v\succeq w$ means
$\geq$ in each component of the vectors $v,w$  in $\Z^{n}$):
$$r_{1}z^{(1)}+\cdots+r_{m}z^{(m)}\succeq z \mbox{ and }$$
$$t_{j}z\succeq z^{(j)} \mbox{ for } j=1,\ldots,m.$$ Then the
$z$-grading on $\R[\underline{X}]$ is covered by the
$z^{(j)}$-gradings.
\end{proposition}
\begin{proof}
We denote by $\deg(f)$ and $\deg^{(j)}(f)$ the degree of a
polynomial $f$ with respect to the $z$- and the $z^{(j)}$-grading,
respectively. First take a polynomial $f$ and suppose $\deg(f)\leq
d$ for $d\in \Z$. So for every monomial $c\underline{X}^{\de}$
occurring in $f$ we have $z\cdot\de\leq d.$ Now for every
$j=1,\ldots,m$,
$$z^{(j)}\cdot\de\leq t_{j}\left(z\cdot\de\right)\leq t_{j}d,$$ so
$\deg^{(j)}(f)\leq t_{j}d.$ Thus $\ps_{j}\colon \Z\rightarrow\Z;
d\mapsto t_{j}d$ fulfills the condition from Definition
$\ref{cover}$.

Now suppose $\deg^{(j)}(f)\leq d_{j}$ for $d_{j}\in\Z$ and
$j=1,\ldots,m$. So for every monomial $c\underline{X}^{\de}$
occurring in $f$,
$$z\cdot\de\leq
r_{1}\left(z^{(1)}\cdot\de\right)+\cdots+r_{m}\left(z^{(m)}\cdot\de\right)\leq
r_{1}d_{1}+\cdots +r_{m}d_{m}$$ holds. So $\ps\colon
\Z^{m}\rightarrow \Z; (d_{1},\ldots,d_{m})\mapsto
r_{1}d_{1}+\cdots+r_{m}d_{m}$ fulfills the other condition from
Definition $\ref{cover}$.
\end{proof}

For example, the usual grading ($z=(1,\ldots,1)$) is covered by
the gradings defined by
$$z^{(1)}=(1,0,\ldots,0), z^{(2)}=(0,1,0,\ldots,0),\ldots,
z^{(n)}=(0,\ldots,0,1).$$ For $n=2$, the two gradings defined by
$$z^{(1)}=(0,1), z^{(2)}=(1,-1)$$ also cover the usual grading.

The following Main Theorem merges the above explained results.
\begin{theorem}\label{all} Let $S\subseteq \R^n$ be a basic closed
semi-algebraic set that contains tentacles $T_{K_j,z^{(j)}}$,
where $K_j$ is compact with nonempty interior and $z^{(j)}\in
\Z^n$ $(j=1,\ldots,m)$. If there exist $r_{1},\ldots,r_{m}\in\N$
such that $$r_{1}z^{(1)}+\cdots+r_{m}z^{(m)} \succ 0,$$ then any
finitely generated quadratic module describing $S$ is stable and
closed. So if $n\geq 2$, such a quadratic module does never have
the Strong Moment Property.

 Such natural numbers $r_i$ exist, if and
only if the only polynomial functions bounded on $$
\bigcup_{j=1}^m T_{K_j,z^{(j)}}
$$  are the reals.
\end{theorem}
\begin{proof}The first part of the theorem is clear from the above
results. We only have to prove the part concerning the bounded
polynomial functions. Note that a polynomial $f$ is bounded on a
tentacle $T_{K,z}$ if and only if it has degree less or equal to
$0$ with respect to the $z$-grading. This follows easily, using
the ideas from the proof of Proposition \ref{xxx}, and the fact
that $K$ is compact and has nonempty interior. So in case there
are natural numbers $r_{1},\ldots,r_{m}\in\N$ with
$$r_{1}z^{(1)}+\cdots+r_{m}z^{(m)} \succ 0,$$ there is no
nontrivial monomial $\underline{X}^{\de}$ that has degree less or
equal to $0$ with respect to all the $z^{(j)}$-gradings. As all
the monomials are homogeneous elements, there can be no nontrivial
polynomial bounded on $ \bigcup_{j=1}^m T_{K_j,z^{(j)}}.$

Conversely, assume there do \textit{not} exists suitable numbers
$r_i$. Then, by a Theorem of the Alternative (see for example
\cite{ad}, Lemma 1.2), there must be $\de\in\N^n\setminus\{0\}$,
such that
$$\de\cdot z^{(j)}\leq 0$$ for all $j$. But this means that the
(nontrivial) monomial $\underline{X}^{\de}$ is bounded on the set
$\bigcup_{j=1}^m T_{K_j,z^{(j)}}.$
\end{proof}

Another class of gradings on the polynomial algebra $A$ is given
by term-orders. A term order is  a linear ordering $\leq$ on
$\N^{n}$ which fulfills
$$\al\leq\be\Rightarrow\al+\ga\leq\be+\ga$$ for all
$\al,\be,\ga\in\N^{n}$. Such a term order extends in a canonical
way to an ordering of the Abelian group $\Z^{n}$. Indeed write
$\ga\in\Z^n$ as a difference $\al-\be$ of elements from $\N^n$;
then define $\ga\geq 0$ if and only if $\al\geq\be.$

We have a grading
$$A=\bigoplus_{\ga\in\Z^{n}}A_{\ga}^{(\leq)},$$ where
$A_{\ga}^{(\leq)}:=\R\cdot \underline{X}^{\ga}$ if $\ga\in\N^{n}$
and $A_{\ga}^{(\leq)}:= \{0\}$ otherwise. We refer to this grading
as the $\leq$-grading. The decomposition of a polynomial $f\in
\R[\underline{X}]$ is
$$f=c_{\ga_{1}}\underline{X}^{\ga_{1}}+\cdots+c_{\ga_{t}}\underline{X}^{\ga_{t}},$$ where
$c_{\ga_{i}}\neq 0$ are the coefficients of $f$ and
$\ga_{1}<\cdots<\ga_{t}$ with respect to the term order. The
degree of $f$ is $\ga_{t}$ then, and the highest degree part is
the monomial $c_{\ga_{t}}\underline{X}^{\ga_{t}}$. Now for these
term order gradings, the question of total stability is easy to
solve. First we apply the reduction result from Proposition
$\ref{modtwo}$ to the generators of the quadratic module. So we
can assume that all the generators have the same degree mod
$2\Z^{n}$. The highest degree parts of the generators are then
monomials $c_{\ga} \underline{X}^{\ga}$, where all the $\ga$ are
congruent modulo $2\Z^{n}$. So obviously the quadratic module is
totally stable if and only if all the occurring coefficients
$c_{\ga}$ have the same sign, and are positive in case the $\ga$
are congruent $0$ modulo $2\Z^{n}$. This gives an easy to apply
method to decide total stability of a quadratic module with
respect to a term order grading.

Note that not all of these $\leq$-gradings induce filtrations with
finite dimensional linear subspaces. For example, a
lexicographical ordering on $\N^{n}$ does not. However, if we
first sort by the usual total degree and then lexicographically,
the subspaces are finite dimensional.

These term order gradings can show stability of quadratic modules,
where the purely geometric conditions derived above and in
\cite{ps} do not apply. So they allow to take into account the
difference between quadratic modules and preorderings.

\section{Examples}\label{examples}

\label{stabex}We start with some examples for the geometric
stability result of Theorem \ref{all}. The first set we look at is
defined by the inequalities $0\leq x,x^2\leq y, y\leq 2x^2$ in
$\R^2$.

 \begin{figure}[h!]
\begin{center}\bf\includegraphics[scale=0.2]{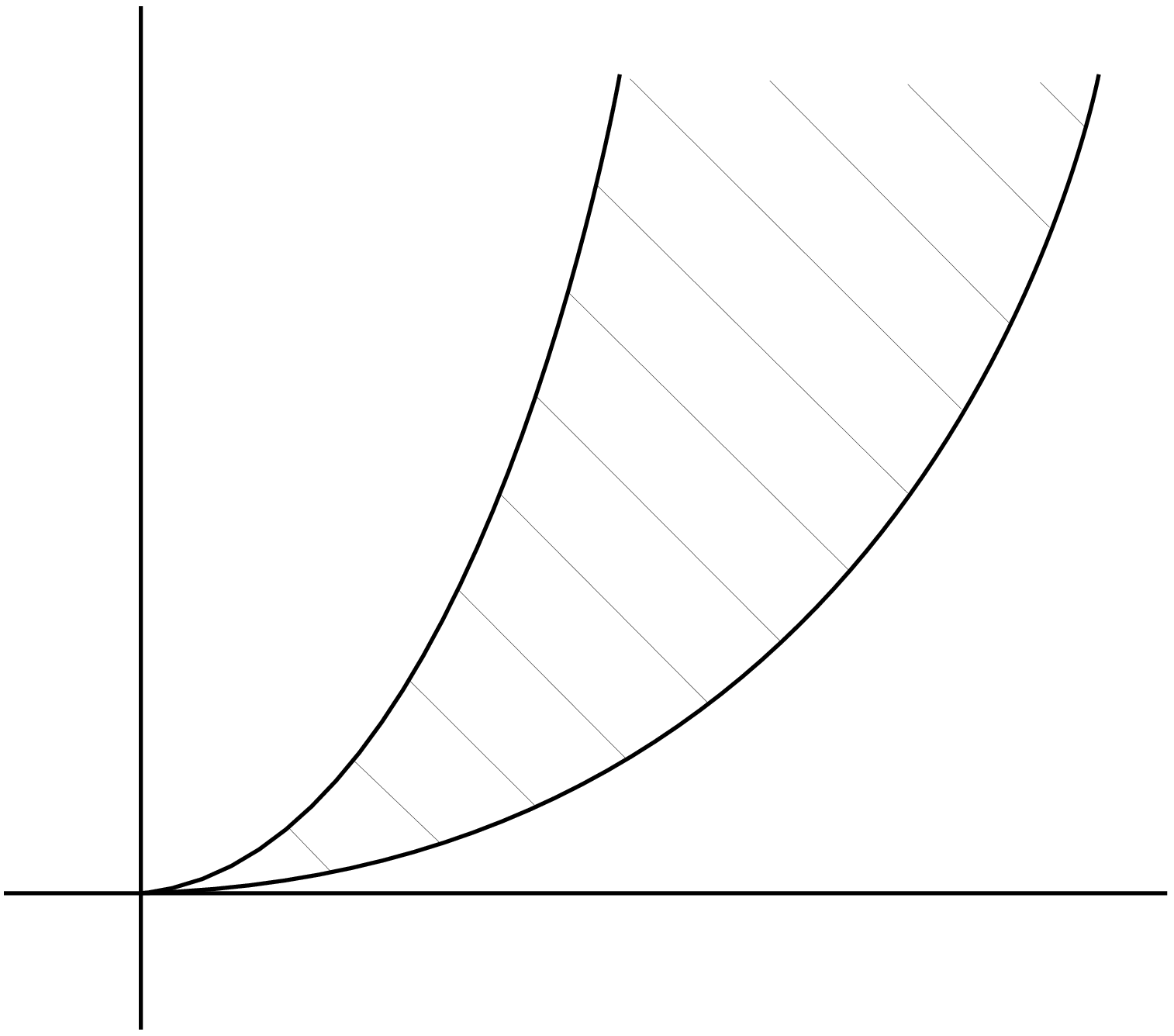}\end{center}
\end{figure}

It contains a tentacle $T_{K,(1,2)}$. Therefore every finitely
generated quadratic module describing this set is stable, thus
also closed and does not have the Strong Moment Property.

The second set is described by $0\leq x, 0\leq y, (x-1)(y-1)\leq
1.$
\begin{center}\bf\includegraphics[scale=0.25]{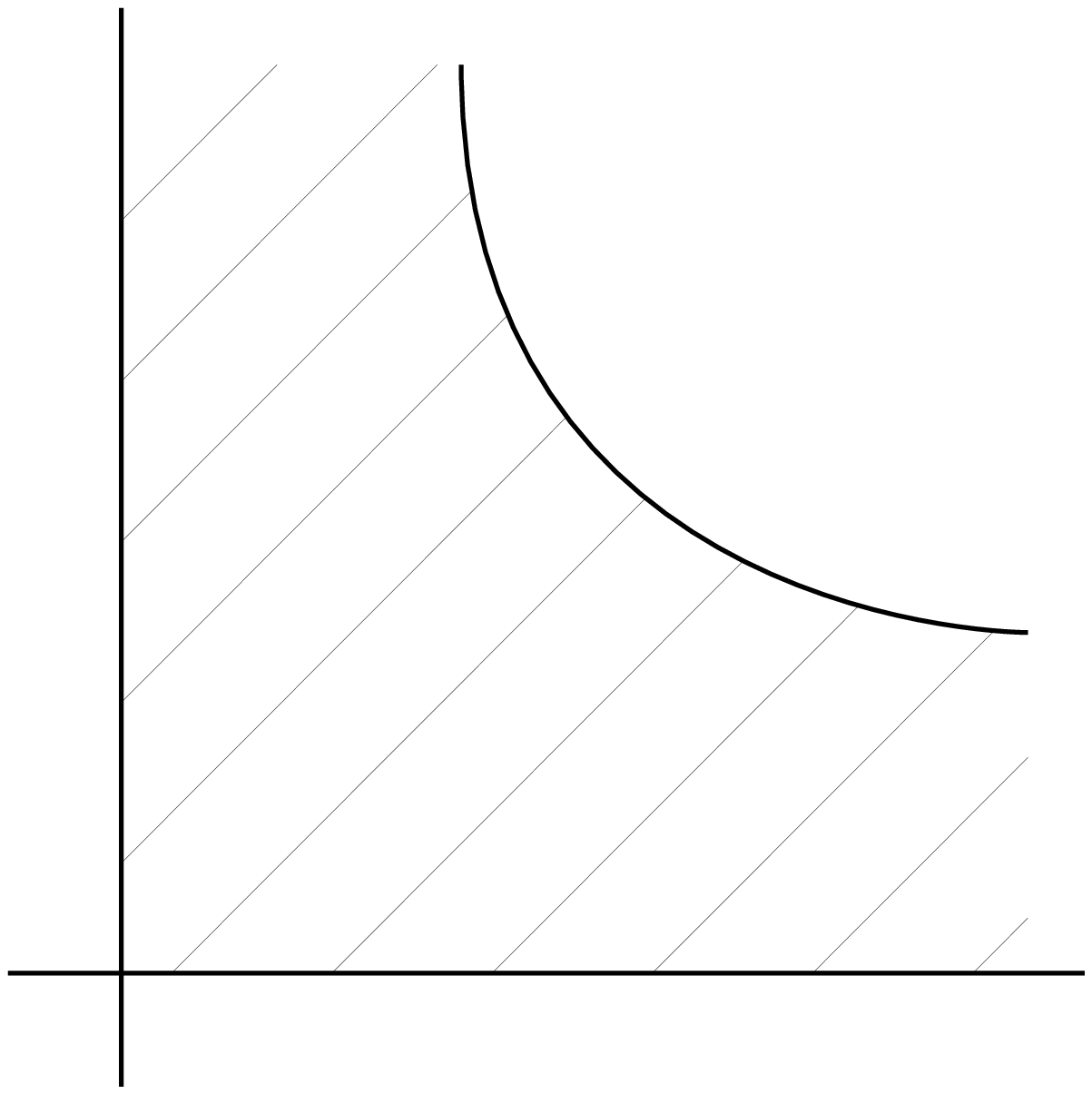}\end{center}
It contains a full dimensional cylinder in each direction of
coordinates (that is, sets $T_{K_1,(1,0)}$ and $T_{K_2,(0,1)}$),
and so every finitely generated quadratic module describing it is
stable, closed and can not have the Strong Moment Property. This
is one way to answer Open Question 4 from \cite{kms}. Another way
to solve this open question is due to Claus Scheiderer
(unpublished). One applies Theorem 3.10 from \cite{ps}.

We can weaken the geometric situation and still obtain stability.
Look at the inequalities $0\leq x, 0\leq y, (x-1)y\leq 1.$
\begin{center}\bf\includegraphics[scale=0.2]{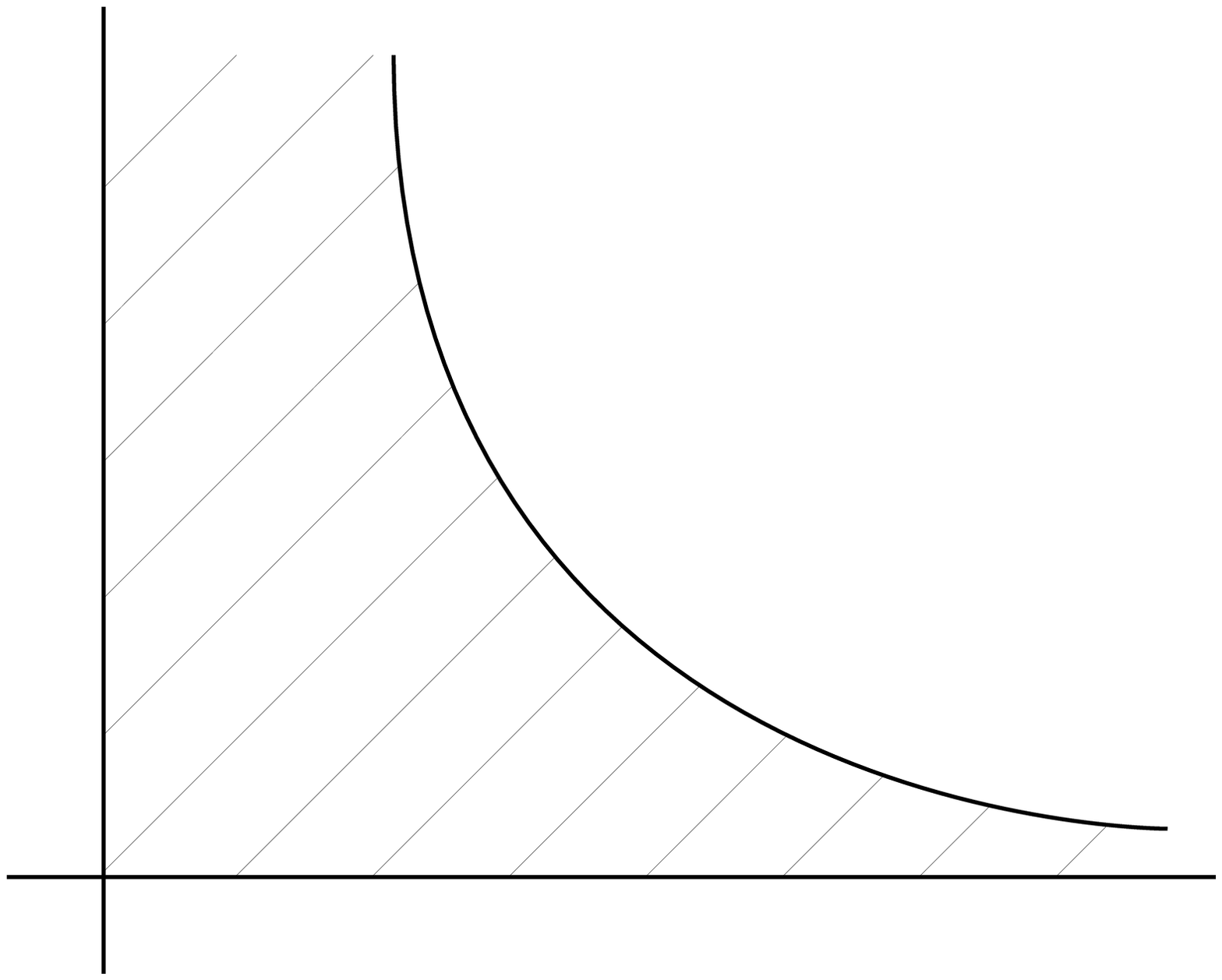}\end{center}
This set contains a full dimensional cylinder in direction of $y$
(a set $T_{K_1,(0,1)}$) and a set $T_{K_2,(1,-1)}$. The $(0,1)$-
and the $(1,-1)$-gradings cover the usual grading, by Proposition
\ref{qcover} (or the fact that there are no nontrivial bounded
polynomials; see Theorem \ref{all}). So every finitely generated
quadratic module describing this set is stable, therefore also
closed and can not have the Strong Moment Property.

We can still go one step further in narrowing the tentacles going
to infinity. Look at the semi-algebraic set defined by $0\leq x,
x^2y\leq 1, -1\leq xy.$
\begin{center}\bf\includegraphics[scale=0.25]{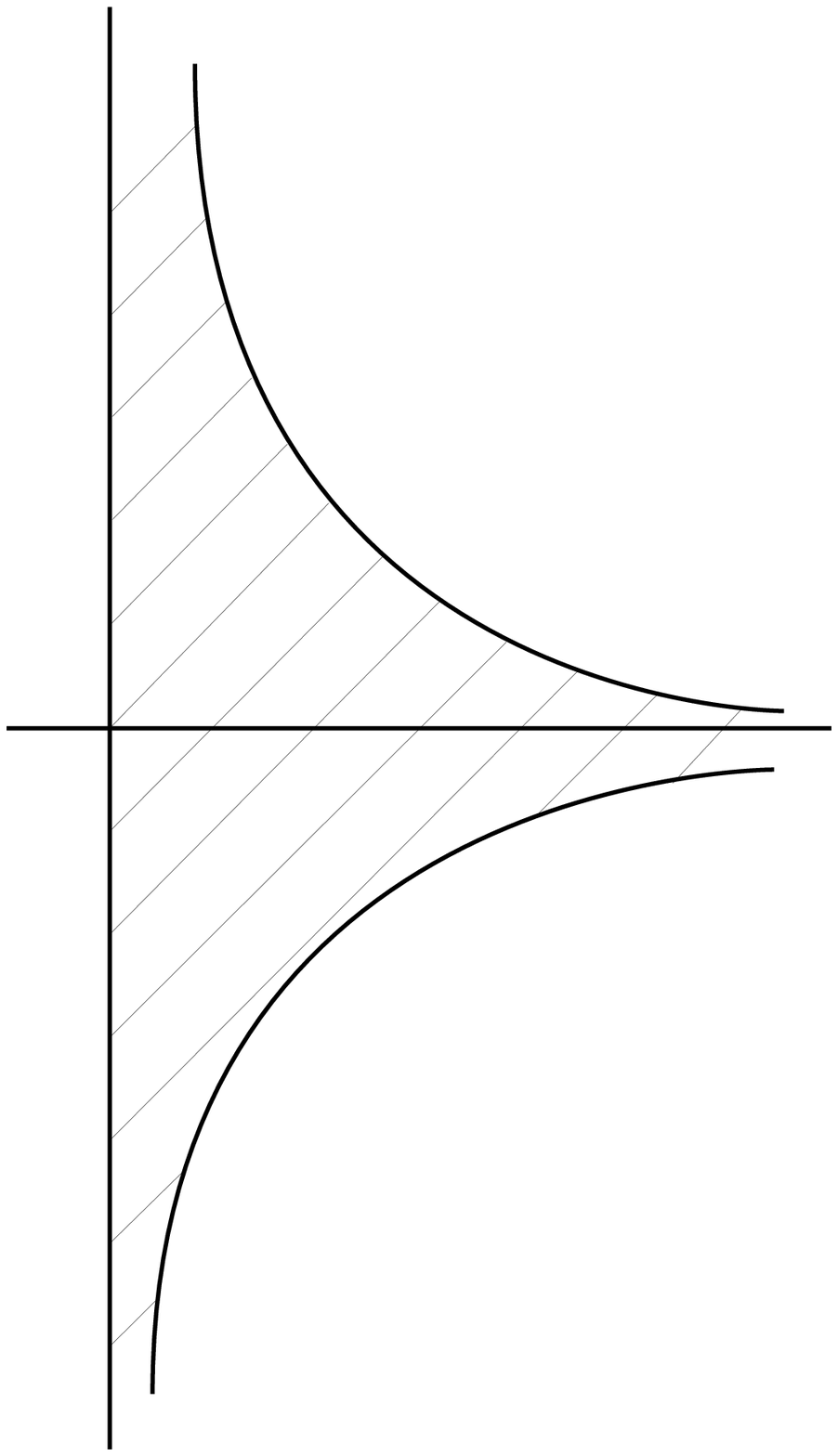}\end{center}
It contains a set $T_{K_1,(-1,2)}$ (corresponding to the tentacle
going to infinity in positive direction of $y$), and a set
$T_{K_2,(1,-1)}$ (corresponding to the part of the tentacle going
to infinity in direction of $x$ that lies below the $x$-axis). As
$$2\cdot (-1,2) + 3\cdot(1,-1) = (1, 1)$$ is positive in
each coordinate,  every finitely generated quadratic module
describing this set is stable, and therefore also closed and does
not have the Strong Moment Property. The considerations also show
that there are no nontrivial bounded polynomials on this set,
which is not completely obvious in this case.

We conclude the section with two non-geometric stability results.
First, look at the semi-algebraic set defined by $0\leq x,0\leq x,
xy\leq 1.$
\begin{center}\bf\includegraphics[scale=0.27]{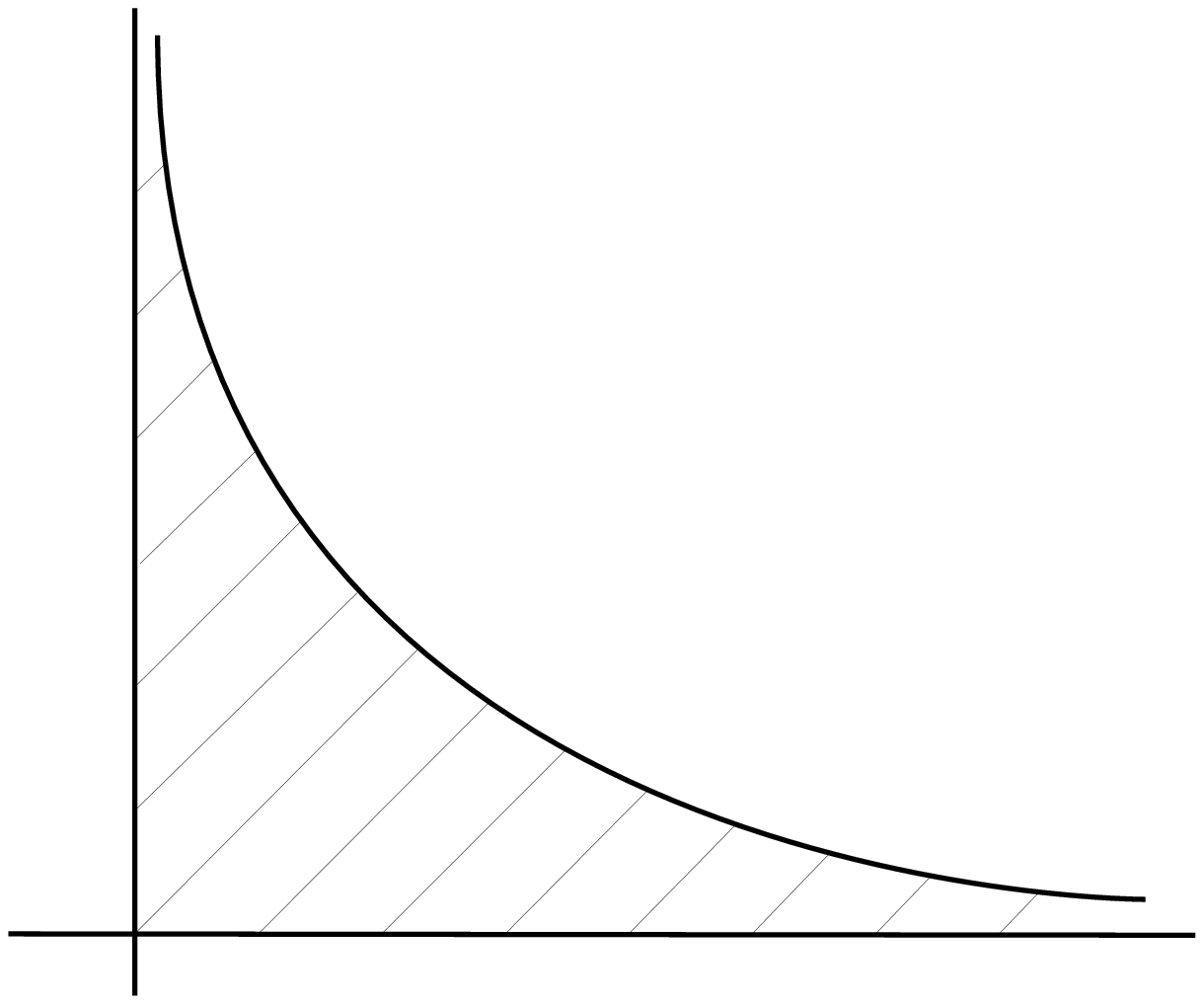}\end{center}
The geometric tentacle result does not apply to this set, and for
example the preordering generated by $X,Y,1-XY$ indeed has the
Strong Moment Property (see \cite{kms}, Example 8.4). So it can
not be stable. However, to the quadratic module
$M_1=\QM(X,Y,1-XY)$
 we can apply the above explained results.
Take the monomial ordering that first sorts by the usual total
degree and then lexicographically with $X>Y.$ No two of the
generators of $M_1$ have the same degree modulo
$2\cdot(\Z\oplus\Z)$. So $M_1$ is stable, closed and does not have
the Strong Moment Property.

Exactly the same argument shows that the quadratic module
$M_2=\QM(X-\frac12, Y-\frac12, 1-XY)$ is stable. In contrast to
$M_1$, it describes a compact set:
\begin{center}\bf\includegraphics[scale=0.35]{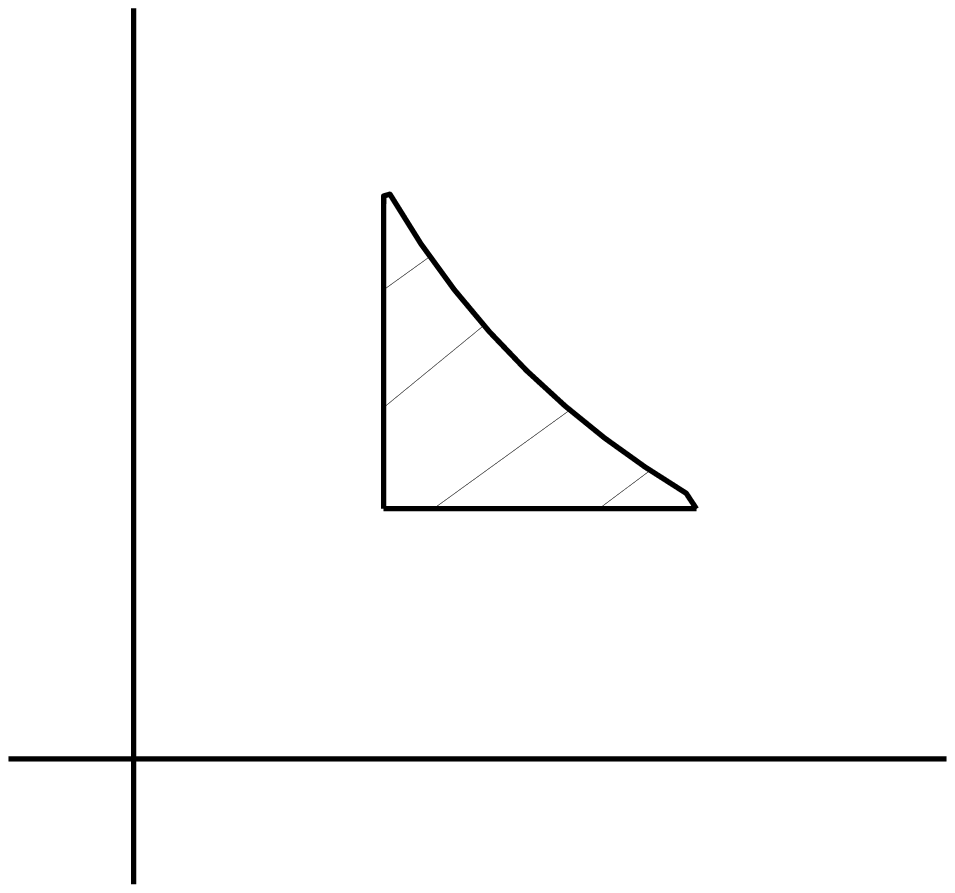}\end{center}
This quadratic module is Example 6.3.1 from \cite{pd}, for a
non-archimedean quadratic module describing a compact set. We can
see here that $M_2$ is not only non-archimedean, but indeed  does
not have the Strong Moment Property, which is stronger.

\section{Strong Stability and the Invariant Moment Problem}

We conclude this section with a generalization of Theorem 6.23
from \cite{cks}. First note that the definition of filtrations and
strongly stable generators of a quadratic module with respect to a
filtration make sense in arbitrary $\R$-algebras. Of course, if
the algebra it not reduced or real, strong stability will only
occur in degenerate situations.

 If $\iota\colon B\rightarrow A$ is
homomorphism of $\R$-algebras, then a filtration on $A$ induces a
canonical one on $B$. If for some $b_1,\ldots,b_s\in B$ the
elements $\iota(b_1),\ldots,\iota(b_s)\in A$ are strongly stable
generators with respect to a given filtration on $A$, then
obviously $b_1,\ldots,b_s$ are strongly stable generators with
respect to that induced filtration.

We now briefly recall the setup of \cite{cks} and refer the reader
to it for more detailed information. Consider a finitely generated
and reduced $\R$-algebra $A$ with affine $\R$-variety $\V_A$.
Denote the set of real points by $\V_A(\R)$. Then $A$ equals
$\R[\V_A]$, the algebra of real regular functions on $\V_A$. Let
$\mathcal{G}$ be a linear algebraic group defined over $\R$,
acting on $\V_A$ by means of $\R$-morphisms. Then
$\mathcal{G}(\R)$ acts canonically on $A=\R[\V_A],$ and if
$\mathcal{G}(\R)$ is compact, the set of invariant regular
functions, denoted by $B=\R[\V_A]^{\mathcal{G}}$, is a finitely
generated $\R$-algebra. So it corresponds to an affine
$\R$-variety $\V_B$ and the inclusion $\iota\colon
B=\R[\V_A]^{\mathcal{G}}\hookrightarrow\R[\V_A]=A$ corresponds to
a morphism $\V_A\rightarrow \V_B$. The restricted morphism
$\iota^*\colon \V_A(\R)\rightarrow \V_B(\R)$ can be seen as the
orbit map of the group action, by a Theorem by Procesi, Schwarz
and Br\"oker. Indeed, the nonempty fibers are precisely the
$\mathcal{G}(\R)$-orbits. Furthermore, for any basic closed
semi-algebraic set $\Se$ in $\V_A(\R)$, the set $\iota^*(\Se)$ is
basic closed semi-algebraic in $\V_B(\R)$. The affine variety
$\V_B$ is denoted by $\V_A / / \mathcal{G}$.

Now suppose $\Se\subseteq \V_A(\R)$ is $\mathcal{G}$-invariant.
Then one can look at the \textit{Invariant Moment Problem} for
$\Se$. That is, one wants to find a finitely generated quadratic
module $M\subseteq \R[\V_A]^{\mathcal{G}}$, such that every linear
functional $L$ on $A$ that is invariant under the action of
$\mathcal{G}(\R)$ and nonnegative on $M$ is integration with
respect to a measure on $\Se$. One of the main results from
\cite{cks} concerning the Invariant Moment Problem is, that this
is possible if and only if $M$ defines $\iota^*(\Se)$ in
$\V_B(\R)$ and has the Strong Moment Property in $B$ (Lemma 6.9 in
\cite{cks}). The situation in $\V_B(\R)$ is often simpler than the
one in $\V_A(\R)$, and so the Invariant Moment Problem can be
solved in cases where the Strong Moment Problem can not.

However, Theorem 6.23 in \cite{cks} yields a negative result about
the Invariant Moment Problem. Roughly spoken, it says that if the
Moment Problem for $\Se$ is not solvable due to some geometric
conditions on $\Se$, then the Invariant Moment Problem is not
solvable either. The result is proven for finite groups
$\mathcal{G}$ and irreducible varieties only. The following result
holds for arbitrary compact groups.

\begin{theorem}\label{inva}
Let the compact group $\mathcal{G}$ act on the affine variety
$\V_A$ and let $\Se$ be a $\mathcal{G}(\R)$-invariant basic closed
semi-algebraic set in $\V_A(\R)$. Fix a filtration of finite
dimensional subspaces of $A,$ and assume that  every finitely
generated quadratic module in $A$ describing $\Se$ has only
strongly stable generators with respect to that filtration. Then
every finitely generated quadratic module in
$B=\R[\V_A]^{\mathcal{G}}$ describing $\iota^*(\Se)$ has only
strongly stable generators with respect to the induced filtration
on $B$ (which consists of finite dimensional subspaces as well).

In particular, if $\dim(\iota^*(\Se))\geq 2$, then no finitely
generated quadratic module in $B$ describing $\iota^*(\Se)$ can
have the Strong Moment Problem. So the Invariant Moment Problem is
not finitely solvable for $\Se$.
\end{theorem}
\begin{proof}
If $\QM(f_1,\ldots,f_s)\subseteq B$ describes $\iota^*(\Se)$, then
$$\QM(\iota(f_1),\ldots,\iota(f_s))\subseteq A$$ describes $\Se =
(\iota^*)^{-1}(\iota^*(\Se))$. This uses that the fibres of
$\iota^*$ are precisely the $\mathcal{G}(\R)$-orbits and that
$\Se$ is $\mathcal{G}(\R)$-invariant. So the assumption implies
that $\iota(f_1),\ldots,\iota(f_s)$ are strongly stable generators
in $A$, and so are $f_1,\ldots,f_s$ in $B$. The result concerning
the Moment Problem follows from \cite{s} now.
\end{proof}
One checks that the geometric conditions from Theorem 6.24 in
\cite{cks} imply, that the conditions from our Theorem \ref{inva}
are fulfilled. Note also that the geometric conditions obtained in
Theorem \ref{all} above  always imply the strong stability of any
finite set of generators for $\Se$. So Theorem \ref{inva} yields a
negative result concerning the Invariant Moment Problem in all of
these cases.


\begin{thebibliography}\markright{}
        \bibitem[Ad]{ad} E.W. Adams: \textit{Elements of a Theory of Inexact Measurements}, Philosophy of Sience 32 (1965), 205-228.


      \bibitem[Au]{au} D. Augustin: \textit{The Membership Problem for Quadratic Modules with Focus on the
        One Dimensional Case}, Doctoral Thesis, University of Regensburg
        (2008).

      \bibitem[CKS]{cks} J. Cimpric, S. Kuhlmann, C. Scheiderer: \textit{Sums of Squares and Invariant Moment Problems in Equivariant Situations}, to appear in Trans. Amer. Math. Soc.

      \bibitem[KM]{km} S. Kuhlmann, M. Marshall: \textit{Positivity, Sums of Squares and the Multi-Dimensional Moment Problem}, Trans. Amer. Math. Soc. 354 (2002), 4285-4301.

      \bibitem[KMS]{kms} S. Kuhlmann, M. Marshall, N. Schwartz: \textit{Positivity, Sums of Squares and the Multi-Dimensional Moment Problem II}, Advances in Geometry 5 (2005), 583-606.

        \bibitem[L]{l} J.B. Lasserre: \textit{Global Optimization with Polynomials and the Problem of Moments}, SIAM J. Optim. 11 (2001), 796--817.

        \bibitem[M1]{m1} M. Marshall: \textit{Positive Polynomials and
        Sums of Squares}, AMS Math. Surveys and Monographs 146, Providence (2008).

      \bibitem[N]{n} T. Netzer: \textit{An Elementary Proof of Schm\"udgens Theorem on the Moment Problem of Closed Semialgebrac Sets}, Proc. of the Amer. Math. Soc. 136  (2008),
    529-537.

     \bibitem[P1]{p1} D. Plaumann: \textit{Stabilit\"at von
    Quadratsummen auf Rellen Algebraischen Variet\"aten},
    Diplomarbeit, Universit\"at Duisburg (2004).

    \bibitem[P2]{p2} D. Plaumann: \textit{Bounded Polynomials, Sums of
    Squares and the Moment Problem}, Doctoral Thesis, University of
    Konstanz (2008).


      \bibitem[PS]{ps} V. Powers, C. Scheiderer:
      \textit{The Moment Problem for nNn-Compact Semialgebraic Sets}, Adv. Geom. 1 (2001), 71-88.

      \bibitem[PD]{pd} A. Prestel, C. N. Delzell, \textit{Positive
      Polynomials}, Springer, Berlin (2001).

      \bibitem[S1]{s} C. Scheiderer: \textit{Non-eEistence of Degree Bounds for Weighted Sums of Squares Representations}, J. Complexity 21 (2005),  823-844.

      \bibitem[S2]{s2} C. Scheiderer: \textit{Sums of Squares on Real Algebraic Curves}, Math. Z. 245 (2003), 725-760.

\bibitem[Sch1]{sch0} K. Schm\"udgen: \textit{Unbounded Operator Algebras and Representation Theory}, Birkh\"auser, Basel (1990).


      \bibitem[Sch2]{sch1} K. Schm\"udgen, \textit{The K-moment Problem for Compact Semi-Algebraic Sets}, Math. Ann. 289 (1991), 203-206.


      \bibitem[Sch3]{sch} K. Schm\"udgen: \textit{On the Moment Problem of Closed Semialgebraic Sets}, J. reine angew. Math. 558 (2003), 225-234.

        \bibitem[Sw1]{sw2} M. Schweighofer: \textit{On the Complexity of
    Schm\"udgen's Positivstellensatz}, Journal of Complexity 20, No. 4
    (2004), 529-543.


    \bibitem[Sw2]{sw1} M. Schweighofer: \textit{Optimization of
    Polynomials on Compact Semialgebraic Sets}, SIAM J. Optim. 15
    (2005), 805-825.


      \bibitem[VB]{vb} L. Vandenberghe, S. Boyd, \textit{Semidefinite
      Programming}, SIAM review 38 (1996), 49-95.

\end{thebibliography}
\end{document}